\documentclass[12pt]{article}
\usepackage{amsthm, amsfonts, url}
\usepackage{amscd,amsmath,amssymb}

\newtheorem{theorem}{\noindent Theorem}

\newtheorem{definition}{\noindent Definition}
\newtheorem{corollary}{\noindent Corollary}
\newtheorem{conjecture}{\noindent Conjcture}
\newtheorem{statement}{\noindent Proposition}

\urldef\vershik\url{vershik@pdmi.ras.ru}
\title{SMOOTH AND NON-SMOOTH $AF$-ALGEBRAS AND PROBLEM ON INVARIANT MEASURES.}

\date{APRIL 1, 2013}

\begin{document}
\maketitle

\abstract{We separate the $AF$-algebras (correspondingly action of the countable groups on Cantor sets)
onto two classes ---- "completely smooth" for which the set
of all indecomposable traces (correspondingly list of all invariant ergodic measures) has nice parametrization,
and the second class --- "completely non-smooth" for which the set of all traces (correspondingly set
of all invariant measures) is Poulsen simplex and therefore there is no suitable parametrization of indecomposable traces,
(ergodic measures) e.g.Choquet boundary.
 Important example of the first type of $AF$-algebra is group algebra of infinite symmetric group, and of the second type
  --- group algebra of some locally finite solvable group. The questions of recognition of those cases are closely related
 to the orbit theory of dynamical systems and theory of filtrations of sigma-fields in measure spaces and Borel spaces.

\bigskip
\bigskip

\textbf{CONTENT}

\smallskip

1. \texttt{INTRODUCTION}

2. \texttt{CENTRAL MEASURES, TRACES, CHARACTERS}

3. \texttt{DEFINITION OF NON-SMOOTH AND COMPLETELY SMOOTH $AF$-ALGEBRAS}

4. \texttt{NON-SMOOTH $AF$-ALGEBRAS AND ACTIONS WITH MANY INVARIANT MEASURES}

5. \texttt{COMPLETELY SMOOTH $AF$-ALGEBRAS}

6. \texttt{CONJECTURE}

\section{INTRODUCTION}
This is the detailed version of the part of my talk on the School on Probability and Statistical Physics
at St.Petersburg (June 2012). I have discussed several problems about the combinatorial and probabilistic structures
which are related to representation theory of the locally finite groups and locally semi-simple
and $AF$-algebras, as well of asymptotic combinatorics and dynamics.  In this note I discuss only one problem of that
list, which I consider as the main one; this problem has deal with the description of the characters of locally finite
groups, traces of  $AF$-algebras, and with invariant measures for ergodic equivalence relations.
It connected with various mathematical areas. Most known example of this type is the problem about the characters of
infinite symmetric group. More general setting is the theory of the central measures on the set of paths of Bratteli
diagram (f.e. Young graph) and the the problem about description of invariant probability measures with respect
to the action of the countable group (for example adic transformation)

The goal this text is to distinguish so called "completely smooth" and "completely non-smooth" cases in the problems
and correspondingly distinguish the $AF$-algebras onto several classes. Our classes can
be formally expressed in terms of $K$-theory, namely, of $K_0$-functor and its dual $K^0$ (\cite{VK}). But we will
not use this language and try to formulate the tools which can be easily calculated. We show some typical examples
of those classes and formulate a Conjecture which
perhaps will help to solve some old problems. In particulary we give the prediction which concerns to the
list of characters for many dimensional Young graphs.

\smallskip
    More concretely, our goal is to investigate the structure of the invariant measures on the Cantor sets,
    rather than explicit formulas for them. We explain the difference between two classes: for non-smooth case
    there are no reasons to expect regular formulas for the traces or invariant measures. But it is important
    to know whether given case is smooth or not and to give a criteria for that. I have started to think about
    two-dimensional (ordinary) Young graph) just after I became familiar with Thoma's result
    and tried to understand when and why such nice answer exists.

  In particular I want to emphasize that the pure combinatorial (or probabilistic) proof even
  for $AF({\frak S}_{\infty})$-algebra (=existense of borel parametrization of the list of indecomposable characters)
  does not exist up-to now. I mean the question here is not about precise formulas for invariant (central) measures
  or traces for $AF$-algebras, which is rather concrete problem, but about the reason why such formula can exist
  and what are the sufficient conditions for existence.

  More deep features of the problem is discovered if we connect it with the orbit theory and with the theory
  of decreasing sequences of sigma-algebras (filtration). Shortly, the difference between smooth and non-smooth cases
  look as a difference between standard and nonstandard semi-homogeneous filtrations.

 We conclude the Introduction with the claim that the problem of the description of the traces of $AF$-algebra
 (characters of the locally finite groups) is the analog of Fourier duality in classical analysis.

  In the second paragraph we explain the problem and the links between traces of $AF$-algebras and invariant measures.
  Then we give the main definitions and in fourth and fifth paragraphs consider non-trivial examples -
  smooth ---  the case of group algebra of symmetric group (Young graph as Bratteli diagram) and non-smooth --- group
  algebra of second order of solvability locally finite group. In the last paragraph we formulate
  our conjecture about smoothness of the case of Bratteli diagram which are Hasse diagram of the countable
  distributive lattices in particular - many dimensional Young graphs.
\bigskip
\bigskip

\section{CENTRAL MEASURES, TRACES, CHARACTERS, ACTIONS.}
 We will mention briefly some notions: Bratteli diagrams, space of the paths, $AF$-algebra, skew-product structure,
 central measures, traces etc. It is easy to find this in the literature (see f.e.\cite{VK} and references their).

 Recall that Bratteli diagram is a locally finite $\Bbb N$-graded graph $\Gamma$ with one initial vertex $\emptyset$,
 with no final vertices and edges (could be multiple with finite multiplcity) which join the vertices of successive levels.
 The important object is the space  $T(\Gamma)$ of all infinite pathes of the graph $\Gamma$ - this sequence of edges (not
 vertices!) started in the from initial vertex without breaks. This space has natural structure of inverse limit
 of the finite sets and consequently has topology of Cantor space. A cylindric complex functions on $T(\Gamma)$
 are the functions which depend on finite parts of the path. The algebra of all cylindric functions on $T(\Gamma)$
 with pointwise multiplication called Gel'fand-Zetlin $GZ(\Gamma)$ this is commutative algebra (see \cite{OV})
 which is maximal if there is no multiplicity of edges. Because $\Gamma$ is $\Bbb N$-graded graph it supplied
 with so called \emph{tail equivalence relation, and correspondingly, tail partition} --- two paths are equivalent
 if they coincided at infinity, more exactly, both paths have the same vertices on the sufficiently large levels.
 Denote this tail partition on the space of the paths as $\xi$. In the interesting cases this is not measurable
 partition. It is possible  to define a transformation of the path space, whose orbit partition coincides
 with tail partition -I mean  \emph{adic transformation} which depend on ordering of the set of edges coming
 to the each vertex. For our goals instead of Bratteli diagrams we can consider Cantor set with filtration
 of all cylindric functions and tail equivalence relation.

 Thus instead of Bratteli diagram and the  space of paths we can consider a standard Cantor space with a filtration
 in the space of all cylindric functions corresponding tail partition and invariant (homogeneous) measures on that
 partition. It happened that we are in the situation with the decreasing sequences of cylindric (measurable)
 partitions in the Cantor set (as topological space), and we can use the analogy with measure-theoretical results
 about decreasing sequences of the partitions (see \cite{V94}). We use this analogy below.

\medskip
  But first of all recall the connection of combinatorics with  $AF$-algebras. It is well-known that each Bratteli
  diagram $\Gamma$ generates canonically a locally-semi simple $\Bbb C$-algebra ${\Bbb C}(\Gamma)$.
  By theorem which independently was proved in \cite{SV,VK} this algebra has natural structure of skew-product
  of Gel'fand-Zetlin subalgebra $GZ(\Gamma)$ and some group $G$ of the tail-preserving transformations of space
  $T(\Gamma)$. We can choose as that group the group $\Bbb Z$, generated by one transformation which is defined
  on the set of almost all paths --- adic transformation \cite{V81}.
The completion of this algebra $C(\Gamma)$ with respect to $C^*$-norm is so-called {\it approximately finite-dimensional
algebra --- $AF$-algebra}. We denote it as $AF(\Gamma)$. The theory of $AF$-algebras considers the connections between
properties of the graph $\Gamma$ itself and algebra  $AF(\Gamma)$, including $K$-theory, classification and
so on \cite{V86,GPS}. But here we consider only very rough difference between  $AF$-algebras.

\begin{definition}
The central measure on the space of paths $T(\Gamma)$ of the graded graph  $\Gamma$ is a Borel probability measure
on $T(\Gamma)$ which is invariant with respect to partition $\xi$. In another words, this is a measure induces
a homogeneous conditional measure on the any measurable subpartition of $\xi$. This is equivalent to the fact
that this measure is invariant under the adic transformation (see \cite{V81}) or any group of the transformations
whose partition on the orbits coincides with partition $\xi$.
\end{definition}

Recall that the trace on $C^*$-algebra $A$ is positive definite linear functional $\varphi: \varphi(hh^*)\geq 0)$,
which has property: $\varphi(hgh^*)=\varphi(g), \varphi(e)=1$ for all $g,h \in A$. If $A$ is a group algebra then
trace define a normalized character on the group.

\begin{theorem}(\cite{SV,VK})
There is a canonical one-to-one correspondence between the traces on the algebra  $C^*(\Gamma)$ and central measure
on $T(\Gamma)$.
\end{theorem}

 This identification is defined by the restriction of the trace on $C^*(\Gamma)$ to the Gelfand-Zetlin subalgebra
 $GZ(\Gamma)$, and trace on this commutative subalgebra is a Borel probability measure on the spectra e.g.
 on the space $T(\Gamma)$. It is easy to check that this measure must be central because of invariance of the trace
 with respect to inner automorphisms. And vice versa, each central measure generates trace as measure-type trace
 skew-product algebras.

The $C^*(\Gamma)$ --- group algebra of the locally-finite group --- is  $AF$-algebra. In this can we can apply
 the theorem.  For example the infinite symmetric group has Young graph of the Young diagrams as branching graph
 of the simple modules and then the central measures are normalized character of the group and vice versa.

We have two identical problems:

\textbf{ To describe the set of ergodic central measures for given Bratteli diagram, or equivalently
to describe all invariant measures for the tail partition of a cylindric filtration of the Cantor set.}
\medskip

and

\medskip

 \textbf{To describe the set of indecomposable traces for given $AF$-algebra.}

As a special case of the last problem we have the following problem:

\medskip
\textbf{To describe the set of the normalized indecomposable characters of the given locally finite group.}

From point of view of representation theory the normalized traces on the $C^*$-algebras generate the representations
of the $C^*$-algebra (group) of finite type in the sense of von Neumann: either type $I_n$ or type $II_1$.
The same is true for normalized characters on the countable groups: if the trace or character is indecomposable
(which means that the central measure is ergodic), then it is factor-representation of the group is of type $I_n$
or type $II_1$ and vice versa.

The decomposable trace or character decomposable can be represent as integral over indecomposable ones,
this decomposition corresponds to central decomposition of the von Neumann algebra of finite type; in terms of
invariant measure it corresponds to ergodic decomposition of invariant measures.

\bigskip

\section{DEFINITIONS OF COMPLETELY SMOOTH AND COMPLETELY NON-SMOOTH $AF$-ALGEBRAS AND GROUP ACTIONS.}

\subsection{Main definition}
 The set of all normalized traces of the $C^*$-algebra is convex compact in the weak topology on the state space,
 (=normalized linear positive functionals on the algebra), moreover it is affine simplex. The extremal points
 of this compact is indecomposable traces which means can't be proper convex combination of the other traces.
 The set of invariant probability Borel measures on the topological compact $X$ under the action of the countable
 group $\Gamma$ is also affine compact simplex and the set of ergodic measures coincide with its Choquet boundary.

 Recall that Poulsen simplex is affine compact simplex in which the Choquet boundary (=the set of extremal points)
 is every where dense in the simplex. Such a simplex is unique up-to affine isomorphism (universality of Poulsen simplex).
 The simplex with closed Choquet boundary called as Bauer simplex, and there are uncountable many non-isomorphic
 Bauer simplecies.

 Denote as $Tr( {\cal A})$ the simplex of all traces on the $C^*$-algebra $\cal A$ and the set of indecomposable traces
 as $ExTr( {\cal A})$ traces (=Choquet boundary of that simplex). In parallel, denote the set of invariant measure on the
Cantor set $K$ with hyperfinite Borel equivalence relation with countable blocks -- $\xi$ ---as  $Inv(K,\xi)$;
 denote the set of the ergodic measures of $Inv(K,\xi)$ as $Erg(K,\xi)$. Partial case -  $\xi$ is the orbit
 equivalence relation for the measure preserving action of the countable group. We consider only hyper-finite $\xi$
 which means that group (if acts freely) is amenable.

\begin{definition}
$AF$-algebra $\cal A$ is called completely non-smooth if the set of normalized traces $Tr( {\cal A})$ as affine 
compact is Poulsen simplex. We say that $\cal A$ is completely smooth if the Choquet boundary $ExTr( {\cal F})$ is open
subset in its weak closure.

 The hyper-finite equivalence relation $\xi$ of the Cantor space $K$ called completely non-smooth if $Inv(K,\xi)$
 is Poulsen simplex.  We say that equivalence relation $\xi$ (or action of the group with this orbit equivalence 
 relation) is completely smooth if the set of ergodic measures (Choquet boundary) $Erg(K,\xi)$ is open subset 
 in its weak closure in $Inv(K,\xi)$. \footnote{We omit some time
the word "completely" in all definitions. It needs strictly speaking only if we consider also the intermediate
case -when $AF$-algebra action of the group) has simultaneously both parts, but we will not consider intermediate
cases here.}
\end{definition}

\subsection{Some remarks}

The condition:  "The set $ExTr(A)$ (correspondingly $Erg(K,\xi)$) is weakly closed in $Tr(A)$
(correspondingly in $Inv(K,\xi))$" e.g. the simplex is Bauer simplex, --- seems too strong for our goals.
In another words we allow the cases which is often ha been appeared when there are some (not so many)
decomposable traces (correspondingly- non-ergodic measures) which is a limit of indecomposable traces (ergodic measures).
In our Conjecture --- see below --- we assume just the case of the weak closure. I do not know name for the simplex
with that property weaker that Bauer property.

\medskip
The definition above put the main problem: how to distinguish those two cases in both situations in concrete examples,
how to verify that given $AF$-algebra (given equivalence relation or action) belongs to one of two classes.
We will illustrate such kind of problems below. The most interesting example concerns of Young diagrams
of different dimensions, distributive lattices and special graphs.

The smoothness of the description above means that we can give a good parametrization of the set of indecomposable
traces (correspondingly - ergodic invariant measures). But of course, to find such parametrization is nontrivial
problem.

\subsection{Connection with other topics: boundary, classification of equivalence relations }

Remark that the definition of the ergodic central measures includes the information about {\it co-transition probability
of the Markov chain which is  unform distribution}. It does not include the transition probabilities of
Markov chain and those probabilities could be very different. In this sense our problem
is NOT the problem of the calculation of Poisson-Furstenberg boundary (see \cite{V2000}),
 rather -close to calculation of Martin boundary. We can say that the problem about description
of central measure is the problem about description of all transition probabilities. But the attempt of direct
calculation of it is not fruitful approach because it is too cumbersome. Centrality of measures (or invariance)
can be generalized to the problem "to find all measures with given cocycle" - in our case this cocycle equal to $1$.
\footnote{In the theory of the stationary Markov compact the central measures called "measure with maximal entropy"}

 In measure theoretical category all ergodic homogeneous hyper-finite partitions are isomorphic (Dye's theorem and lacunary isomorphism theorem).
In topological category it is not true:we have various tail partitions which are not (even Borel) isomorphic; they have the different structure
of the set of invariant measures. The result in \cite{AK} claim that the cardinality of the number of invariant
measures is complete invariant of Borel hyperfinite minimal equivalence relation upto Borel equivalence
\footnote{I am grateful to professor A.Kechris for this reference.}. So we need in formulation intermediate
category between measure-theoretical and borel in which we can take in account some properties of the simplex
of traces (invariant measures). $K$-theory by definition gave us such possibility for $AF$-algebras, but
without concrete tools and this is the main difficulty in our problem.

\subsection{Filtration an standardness}
It seems that more important link of the problem concerns to the theory of filtration.
In our situation we have the following useful fact:

\begin{statement} (\cite{VK})
Each ergodic central measure on the $N$-graded graph $\Gamma$ is a Markov measure with respect to the graded structure
of graph.
\end{statement}

So we must consider the tail (hyperfinite) equivalence relation and its Markov approximation. This is the question of the theory of
the decreasing sequence of finite partitions (filtration). More exactly we have as a Cantor set
(space of paths of the Bratelli diagram) a Markov (non-stationary)
  compact $\cal X$ with finite state spaces
on each coordinate: $${\cal X}=\{\{x_n\}_{n=0}^{\infty}; x_n \in X_n, |X_n|<\infty, m_{x_n,x_{n+1}}=1, \quad
 M_n=\{m^n_{i,k}\}\in Mat_n\{0;1\}, n \in \Bbb N. \}$$

We have a filtration of the sigma-fields: $\{{\frak A}_N\}$ (which are generated by coordinates with number greater than $N$)
Conjecturally we can formulate the answer on the question about smoothness of  of the tail equivalence relation in terms
standardness of that filtration. Recall \cite{V94}, that stationary Markov filtration is standard. The same is true for
the case of group algebra of ${\frak S}_{\infty}$ -see theorem 2.

\bigskip

\subsection{Link with general theory of $C^*$-algebras}

The notions which was defined make sense for the arbitrary $C^*$-algebras and arbitrary equivalence relations as actions.
So, we can compare our definition with analog in the theory of $C^*$-algebras --- of the sharing of the $C^*$-algebras
on liminaire and anti-liminaire (see \cite{Dix} corresponds to the sharing of the  $AF$-algebras onto

1)$AF$-algebras for which description of the indecomposable traces  is "smooth problem"
(this class contains all liminaire $AF$-algebras)

and

2)other $AF$-algebras (which must be anti-liminaire). Equivalently, we divide the problem of the description
of ergodic invariant measures onto two classes.

It is naturally to compare this question with the question about classification of irreducible representations
of groups or algebras. It is well-known (see\cite{Dix}) that the problem of the classification of all irreducible
$*$-representation for $C^*$-algebra or group up-to unitary equivalence or? equivalently -- the description
of the pure states of the algebra up-to equivalence, --- is tame or "smooth" iff this algebra so called "liminaire".
By classical results by Glimm the problem is wild (not smooth) for "anti-liminaire" algebras.  For countable groups
the analog of this by Thoma's theorem is the following: for the groups which are finite extension of commutative
groups (=virtually commutative) and only those groups the problem of the description of the characters is tame (smooth).
The notion of the smooth and non-smooth $AF$-algebras allow to consider the further classification of
"anti-liminaire" algebras. So the meaningful question is: for what kind of anti-liminaire algebras
(or non virtually commutative groups) the description of the traces (characters) is smooth problem?

\section{EXAMPLES OF THE COMPLETELY SMOOTH CASE: YOUNG DIAGRAMS AND YOUNG GRAPH.}

 It is very easy to give example of completely smooth $AF$-algebra or smooth hyperfinite equivalence relation -
 For example Glimm algebra (UHF) like $\bigotimes_1^{\infty}M_2 {\Bbb C}$ has only one trace,
 and tail equivalence relation in the Cantor space $\prod_1^{\infty} Z/2$ has only one invariant measure (Haar measure).
 In this case  the filtration of tail partitions is standard diadic sequence (\cite{V94}). The Pascal graph (\cite{V11})
 as Bratteli diagram defines well-known $AF$-algebra which is also smooth. Perhaps the Pascal filtration is also standard.
 Here we give more complicate examples and to prove smoothness. Namely we consider the group algebra symmetric group which
 is very popular now and became years ago the starting point for further investigations. The cases of infinite
 unitary and orthogonal groups are also smooth in our sense, --- for all these groups we also have full list
 of the characters.

\smallskip
 Concerning to Young graph that problem about characters of infinite symmetric group had several solutions:
 the first one was pure analytical solution by E.Thoma (\cite{Th}) who posed firstly the question.
 His solution contained noting about combinatorial or group-type feature of the problem.
 The second proof by A.Vershik-S.Kerov (\cite{VK}) used group approximation; the third approach  due to A.Okounkov (1992),
 based on operator theory method. It is interesting that the problem itself in all this cases had considered
 from different point of view: Thoma's formulate it as a problem about description of the positive definite functions
 on infinite symmetric group,  Vershik-Kerov approach was about central measures on the space of the infinite paths
 of Young graphs, the approach suggested by Ol'shansky \cite{Ol} and realized by A.Okounkov \cite{Ok} concerned
 so called admissible representations of the that group. We have used with S.Kerov my ergodic method which gives
 a general approach to similar problems but needs in complicate calculations.

We start with canonical example of the completely smooth case -group algebra of infinite symmetric group.
The group algebra is of course anti-liminaire (no smooth classification of the irreducible representations)
but completely smooth -list of traces has nice parametrization.

Consider Young graph $\textbf{Y}$ as Bratteli diagram of the group algebra of the infinite symmetric group =
${\Bbb C}({\frak S}_{\infty})$. The vertices of the Young graph $\textbf{Y}$ is Young diagrams and number
of cells is graduation; the finite path in Young graph $\textbf{Y}$ is ordinary Young tableaux;
we will consider infinite paths which are infinite Young tableau. On the space of all infinite paths $T(\textbf{Y})$
we have tail equivalence relation, and we can define the set of central measures on  $T(\textbf{Y})$,
which, as we saw, in the canonical correspondence with space of traces on ${\Bbb C}({\frak S}_{\infty})$,
or characters on ${\frak S}_{\infty}$. Ergodic measures corresponds to indecomposable traces (characters).

As we have mentioned the list of indecomposable characters $Char({\frak S}_{\infty})$ of infinite symmetric 
group was found by Thoma \cite{Th} in pure analytical way (see formula below).
Ergodic method \cite{VK} in direct formulation for this case asserts that any ergodic central measure $\mu$
on the space $T(\textbf{Y})$ is weak limit of the sequence of the elementary invariant measures
$m_n(\lambda_n)$ where the last measure is uniform distribution on the the set of all Young tableaux with given
Young diagrams $\lambda_n)$ with $n$ cells. In order to find all ergodic central measures we must consider
all possible sequences of Young diagrams
$\{\lambda_n\}_{n=1}^{\infty}$. The problem is to choose enough such sequences of $\lambda_n)$ in order to obtain
all central measure, and then to make calculations of the weak limits for the obtaining of final Thoma's formula.

The key argument in the proof of \cite{VK} was that it is enough to consider the sequences of $\lambda_n)$
which have frequencies of the growth of rows and columns in the diagrams:
$$\alpha_k=\lim_n\frac{r_k(\lambda_n)}{n}; \quad k \in {\Bbb Z}\setminus 0,$$

where $r_k(\lambda)$ is length of $k$-th row (when $k>0$) and length of $-k$-column (when $r<0$) 
of the diagram $\lambda$.

The explicit formula of the characters (traces) includes only sequence 
$\alpha =\{\alpha_k\}_{k \in {\Bbb Z}\setminus 0}$,
 and looks as a function $\chi$ on the group as follow:

$$ \chi_{\alpha}(g)=\prod_{n>1} s_n(\alpha)^{c_n(g)},$$

where $c_n$ is (finite) number of cycles  of length $n>1$ in the (finite) permutation $g$,
and $$s_n=\sum_{k \in {\Bbb Z}\setminus 0} {sgn(k)}^{(n-1)}\cdot {\alpha_k}^n, \quad n>1$$ (super newtonian sum).

Here $$\alpha=\{\alpha_k\}_{k \in {\Bbb Z}\setminus 0}: \alpha_1\geq \alpha_2 \geq \dots \geq 0, \quad \alpha_{-1}\geq \alpha_{-2} \geq \dots \geq 0; \quad \sum_{k \in {\Bbb Z}\setminus 0} \alpha_k \leq 1$$
Interpretation of the value $ \chi_{\alpha}(g)$ is a measure of fixed points of the action of the element $g$ 
in appropriate measure space depending of $\alpha$ (see \cite{V12}). Most important case is the case 
of $\alpha_k\equiv 0$; in the case $ \chi_{\alpha}(g)=\delta_e$ (delta function at identity element of the group), 
and corresponding central measure on the space of Young tableaux  $T(\textbf{Y})$ is {\it Plancherel measure}. 
The corresponding representation is regular representation of the group ${\frak S}_{\infty}$.

An important observation.

\begin{theorem}
Consider any central measure as measure on Markov compact of the paths (Young tableaux) and look at the
sequence of the sigma-field ${\frak F}_n, n=1\dots$ ( or measurable partitions) generated by the coordinate
of the paths with number more that $n$. Than this sequence (as homogenous partitions)
is standard in the sense of \cite{V94}.
\end{theorem}

It is difficult to explain all the notion here, but the notion of the standard-ness is easy to formulate
when we consider it in the framework of measure theory: the sequence of homogeneous partition
 is standard if normalized entropy equal to zero. As we had noticed in this case it is possible
 to connect the standardness with smoothness.

\section{NON-SMOOTH $AF$-ALGEBRAS AND ACTIONS.}

 The following skew-product gives the example of the "non-smooth" situation: let us consider the group
 $$G=\sum_1^{\infty} Z/2;\quad G=\bigcup G_n; \quad G_n=\sum_1^n Z/2$$ and its  Bernoulli action group $G$
 on $X=2^G=\prod_{g \in G}Z/2 $. The semi-product of the group $G$ and space $X$ is a group
 $$\hat G = G\rightthreetimes 2^G$$ which is locally compact solvable group.
 More convenient instead of $X$ to consider its dense countable subgroup $X_0=\sum_{g \in G}Z/2$
 and the semi-product of the group $G$ and space $X_0$, e.g. group $${\bar G}
 =G\rightthreetimes\sum_{g \in G}Z/2$$ with natural action of $G$ on $X_0$.

 Of course $\bar G$ is locally finite countable solvable group, its group algebra $C(\bar G)$ is locally
 semi-simple algebra, and $C^*(\bar G)$ is its $C^*$-group-algebra, which is an $AF$-algebra.
 Each finite dimensional simple subalgebra is a matrix algebra generated by the irreducible representations
 of $G_n$ on $2^{G_n}$ for some $n \in \Bbb N$. Remark that matrix algebra is semidirect product of the action
 of $G_n$ on the group of character $2^{G_n}$). The result of this is the following theorem:

 \begin{theorem}
 The Bratteli diagram $\Gamma(\bar G)$ of the group algebra $C(\bar G)$ has the following form; the set of vertices
 (simple modules) of the graph of the level $n$ are parameterized by the orbits of the natural action of group
 $G_n$ on the space $2^{G_n}$. The edges of the graph between orbits $O$ of the level $n$ and orbit $O'$ of
 the level $n+1$ of the diagram has multiplicity $0$,$1$ or $2$ as follow: recall that the orbit $O'$
 is non ordered union of two orbits of the level $n$; so multiplicity equal zero, if orbit $O$ is not part
 of the orbit of $O'$; if $O$ is a part of the orbit of $O'$, say $O=O'\cup O''$ then the multiplicity 
 equal $1$ if $O' = O''$, and $2$ if $O'\neq O''$.
 \end{theorem}

 It is easy to conclude from this that the space of all paths in the graph $\Gamma(\bar G)$ coincides with group
 $X=2^G$ and partition on the orbits of the action of $G$ on $X$ are the same as the tail partition on the paths.
 The combinatorics of these graph in terms of combinatorics of diadic decreasing sequences of measurable partitions
 has been studied in \cite{V94}.

 This description show us what is the Gelfand-Zetlin algebra GZ of $C(\hat G)$. Observe that the canonical structure 
 of semi-direct product of $AF$-algebra does not coincide in general with the structure of semi-direct product 
 on the group, but in our case this is true. More exactly Gelfand-Zetlin subalgebra $GZ$ of $C(\hat G)$ does 
 coincide with Group algebra of the group $X=2^G$.
  \begin{theorem}
  There is a bijection between the list of the traces of the $AF$-group algebra $C(\bar G)$ of the group $\bar G$
  and the list of Borel probability measures on the space $X=\prod_{g \in G}Z/2$ with respect to the natural action 
  of the group $G$.
  \end{theorem}

   \begin{corollary}
  The $AF$-algebra $C(\hat G)$ of the group $\hat G$ is non-smooth in the sense of the definition which was given above.
   \end{corollary}
   \begin{proof}
   The simplex of the invariant measures on the space $X=\prod_{g \in G}Z/2$ with respect to action of the group $G$
   is Poulsen simplex because there are everywhere dense set $\delta$ of the discrete invariant measures with finite 
   support, namely, measures whose support is finite set of the characteristic functions of the cosets over subgroups 
   of finite index.
   \end{proof}

   The fact that the action of the for any countable group $G$ which has no Kazhdan property $T$ the simplex of the
   invariant measures with respect to action $G$ in $2^G$ is Poulsen simplex was proved in \cite{GW}. For the group
   $G=\sum_1^{\infty} Z/2$  this is evident. But now we need to identify that simplex with the simplex of the traces
   on group algebra  $C(\hat G)$. Let us identify the central measures as the invariant measures on the group $X$
   with traces and characters on the algebra and group. It is evident that the regular representation of the group 
   $\bar G$ corresponds to Haar (Lebesgue) measure on $X$ and the action of the group $G$ is Bernoulli action 
   of this group with Bernoulli measure $\mu$. Moreover the regular representation in usual formulation 
   (in the space $l^2(\bar G)$) is isomorphic to von Neumann model of $II_1$ factor - as representation of 
   semi-direct product $G\rightthreetimes 2^G$ (with left and right action of the group $\bar G$. A very good 
   exercise is to calculate the {\it Plancherel measure} as measure on the space of orbit, and its when order of 
   orbits tends to infinity.

 Now let us interpret this measure in terms of Markov compact of the sequences of orbits (= edges of Bratteli
 diagram; remember that central measure is a measure on the paths e.g. sequences of edges of that diagram).
 Crucial fact is the following:
 \begin{theorem}
 Consider the image of invariant Bernoulli measure $\mu$ on the space $X=2^G$ as a measure on the paths 
 of the Markov compact (see definition in paragraph 2). Consider the decreasing sequences 
 $\{\xi_n\}_{n=1}^{\infty}$ of the sigma-fields (measurable partitions) where $\xi_n$ is sigma-algebra 
 generated by the coordinates with numbers more than $n$n (correspondingly partition on the paths with the 
 same $n$-tail -- coordinates wit numbers more that $n$)
 Than decreasing sequences $\{\xi_n\}_{n=1}^{\infty}$ is not standard in the sense of \cite{V94}. Moreover, this is
 is example of the decreasing sequences with positive entropy in that sense.
  \end{theorem}

  It is enough to mention that this example literally the same as in the paper \cite{V94},
   but it was considered at that paper in the framework of completely different reasons
   (theory of diadic sequences of measurable partitions, no any groups, $AF$-algebras and so on).

\section{THE CONJECTURES AND THE PROBLEMS.}
 \bigskip

  Classification of $AF$-algebras by the theorem by Elliott and others (\cite{El,ElH}) was reduced to the calculation of the $K_0$-
  functor (Grothendick group of the classes of projective modules) with additional structure of Riesz group
  (ordering by the cone of true modules), and fixed element (one-dimensional free module).
  Sometimes it called as "dimension group"). Unfortunately, it is difficult to apply this very important theorem
  because the calculations are very complicate.
  F.e. for group algebra of infinite symmetric group the answer which was done in \cite{VKK}, is rather cumbersome.
  The answer on the question whether  $AF$-algebra is completely smooth or not of course is contained
  in the $K_0$-functor, but we need to extract it and it is not so easy.\footnote{Remember that the list of traces 
  (invariant measures) gives only a part of invariants of $AF$-algebra, because  $K_0$-functor could have 
  the infinitesimal modules part of which does not separate by the traces.}

  Perhaps, it is difficult to give necessary and sufficient conditions for $AF$-algebras (or for actions of the groups)
  which guaantee the belonging to completely smooth class or to other classes. It is important to have at least
  sufficient conditions for smooth case. Hopefully such a condition gives us the hope to find precise formula 
  for all characters and invariant measures.

   $K$-theoretical approach to $AF$-algebra can be considered as {\it approximating approach} with finite dimensional algebras.
   But there is a dual ("co-approximating") approach, which means that instead of increasing sequences of finite
   dimensional algebras we consider decreasing sequence (filtration) of infinite dimensional algebras.

    In combinatorial term it means that we consider the sequence of the tail sigma-fields of sets of the paths which
    coincide after $n-th$ level, $n=1, \dots$. For given Bratteli diagram $\Gamma$ define this filtration
    (=decreasing sequence of the co-finite partitions)
    as ${\frak A}_n(\Gamma), \quad n=1 \dots.$. For two different Bratteli diagrams of the given $AF$-algebra
    it is easy to formulate what kind of isomorphism of corresponding filtration. Let us call it {\it lacunary isomorphism};
    and consider all such filtrations as filtrations on the standard Cantor space. The intersection of the partitions
    is {\it hyperfinite tail equivalence relation} about which we have discussed in the previous paragraphs.

    \textbf{PROBLEM} To classify the filtrations on the Cantor space up to lacunary Borel equivalence.

    {\it More or less clear that this problem is "wild" in the sense of classification theory. But $K_0$-theory
    shows that the classification can be reformulate in other terms. More important that wildness does
    not contradict to have constructive answer on some natural questions like completely smoothness or non-smoothness.}

    The analog of this Problem in measure theoretical category was done in the papers by author (\cite{V94})
    one of the main result is lacunary theorem which is claim that two ergodic diadic filtrations are lacunary
    equivalent up to measure preserving transformation. But in Borel category it is not true. Moreover as we mentioned
    (in paragrph3.3) that even intersection of the filtration (corresponding equivalence relation) has additional
    invariant (number of invariant measures) see \cite{AK}. But the isomorphism of the equivalence relation is too
    rough because the set of invariant measures could be infinite, nevertheless the $A$-algebras could be smooth or
     non-smooth cases which are not isomorphic in our sense.

    \smallskip

    So we need to develop the analog of the measure-theoretical approach to the filtration in Borel category.
    The main tool in that theory is notion of standardness and criteria of standardness (see \cite{V94})
    That approach was developed mainly for homogeneous partition (like diadic) which means that the blocks of
    $n-th$ partitions in the filtration has the same number of points in almost all blocks and the conditional
    measure is uniform. It is important to generalize the main tools to the case semi-homogeneous partitions
    - the uniform conditional measure on the blocks (centrality), but number of points is arbitrary finite.
   Hopefully, these tools can help to obtain the information about invariant measures and traces of $AF$-algebras.
     Our example of the completely non-smooth action (paragraph 5) shows the connection of non-standardness of the
     tail sequence and non-smoothness of$AF$-algebra.

    \smallskip

    I will consider the analog of measure-theoretical approach elsewhere but here I formulate a Conjecture
    of this type for very concrete case.

  Consider a countable distributive lattice $\Gamma$; by well-known G.Birkhoff's theorem there exist such
  a poset (=partial ordered set) $\zeta$ for which $\Gamma$ is the set of all finite ideals with usual
  order $\Gamma=ID_f(\zeta)$. F.e. Young graph $\Bbb Y$ is a Hasse diagram  for the distributive lattice
  of finite ideals of the poset ${\Bbb Z}_+^2$ \footnote{Each distributive lattices is a ${\Bbb N}$-graded poset; Hasse diagram of the graded poset is graded graph vertices of which is elements of poset and edges joins to elements on of which directly preceded to the second element;  Hasse diagram of the graded poset can be considered as Bratteli diagram of some $AF$-algebra.}. The generalization of Young graph is the lattice of finite ideals of the poset  ${\Bbb Z}_+^n, n=3,\dots$ -its Hasse diagram is "Young" graphs of higher order --  ${\Bbb Y}^n$. We can consider the $AF$-algebra which corresponds to is ${\Bbb Y}^n$ --- $A({\Bbb Y}^n)$. For $n>2$ this algebra is not group algebra of any group.
  Suppose $A(\Gamma)$ is $AF$-algebra corresponded to the Hasse-Bratteli diagram of distributive lattice $\Gamma$.

  \begin{conjecture}
 The list of indecomposable traces of $AF$-Algebra  $A(\Gamma)$ is weakly closed and consequently $AF$-algebra
 is completely smooth; In particular the same is true for $A({\Bbb Y}^n)$.
  \end{conjecture}

Possible proof of this Conjecture must use the list of infinite minimal ideals of the corresponding posets  ${\Bbb Z}_+^n$ which give the "frequencies" of minimal ideals for these measures. It is not difficult to prove the existence of frequencies, but the problem is to prove that the set of frequencies uniquely defines central measure and consequently the trace. One can apply combinatorial technique of ordinary Young diagram (contents etc.) but there is no tools like Symmetric Functions and no good description of the corresponding $AF$-algebra.

For $n=2$ the Conjecture about Young Graph $Y$ is true -- this is equivalent to the Thoma's theorem \cite{T}
in our formulation \cite{VK}. For $n=3$ the equivalent conjecture was considered by me with S.Kerov in 80-th.

The last remark - all this problems concern to the limit behavior of the product of adjacent matrices of
Markov compact of the paths in Bratteli diagrams; this is the link with the typical problems in statistical physics.

\end{document}